\documentclass[11pt, letterpaper]{article}

\usepackage[utf8]{inputenc}
\usepackage{amsmath, amssymb, amsthm, stmaryrd, dsfont, mathrsfs}
\usepackage{ytableau}
\usepackage{fullpage}
\usepackage{graphicx}

\newtheorem{thm}{Theorem}
\newtheorem{prop}[thm]{Proposition}
\newtheorem{cor}[thm]{Corollary}

\theoremstyle{definition}
\newtheorem{defn}[thm]{Definition}

\newtheorem{obs}[thm]{Observation}
\newtheorem{remark}[thm]{Remark}
\theoremstyle{remark}

\newcommand{\flr}[1]{\lfloor #1 \rfloor}
\newcommand{\clg}[1]{\lceil #1 \rceil}
\newcommand{\N}{\mathbb{N}}

\newcommand{\vcenteredinclude}[2]{\begingroup
\setbox0=\hbox{\includegraphics[#1]{#2}}%
\parbox{\wd0}{\box0}\endgroup}

\title{Ramsey Numbers of Interval 2-chromatic Ordered Graphs}
\author{
Dana Neidinger\thanks{Department of Mathematics, University of Illinois,
dn2@illinois.edu}\: and
Douglas B. West\thanks{Departments of Mathematics, Zhejiang Normal University
\& University of Illinois, dwest@math.uiuc.edu.  Research supported by
Recruitment Program of Foreign Experts, 1000 Talent Plan, State Administration
of Foreign Experts Affairs, China.}
}
\date{February 2018}

\begin{document}

\maketitle

\noindent
\normalsize

\begin{abstract} 
An \emph{ordered graph} $G$ is a graph together with a specified linear ordering on the vertices, and its \emph{interval chromatic number} is the minimum number of independent sets consisting of consecutive vertices that are needed to partition the vertex set. The \emph{$t$-color Ramsey number} $R_t(G)$ of an ordered graph $G$ is the minimum number of vertices of an ordered complete graph such that every edge-coloring from a set of $t$ colors contains a monochromatic copy of $G$ such that the copy of $G$ preserves the original ordering on $G$.
An ordered graph is \emph{$k$-ichromatic} if it has interval chromatic number $k$. We obtain lower bounds linear in the number of vertices for the Ramsey numbers of certain classes of 2-ichromatic ordered graphs. We also determine the exact value of the $t$-color Ramsey number for two families of 2-ichromatic ordered graphs, and we prove a linear upper bound for a class of 2-ichromatic ordered graphs.  
\end{abstract}

\section{Introduction}

An \emph{ordered graph} $G$ is a graph together with a specified linear ordering on the vertices. We can view the vertices as placed along a horizontal line. 
An \emph{interval coloring} of an ordered graph $G$ is a partition of $V(G)$ into independent sets of consecutive vertices (intervals); we call these sets \emph{parts}.
The \emph{interval chromatic number} of an ordered graph $G$, written $\chi_i(G)$, is the minimum number of parts in an interval coloring of $G$. When the interval chromatic number is $k$, we say that $G$ is \emph{interval $k$-chromatic} or simply \emph{$k$-ichromatic}. 
Note that since the parts of an interval coloring are independent sets, always $\chi_i(G)$ is at least the chromatic number.

The $t$-color \emph{Ramsey number} of a graph $G$ is the minimum number of vertices of a complete graph such that every edge-coloring from a set of $t$ colors contains a monochromatic copy of $G$. An ordered graph $G$ is a \emph{subgraph} of an ordered graph $G'$ if there is an order-preserving injection from the vertices of $G$ to the vertices of $G'$ that preserves edges; we then also say that $G$ is \emph{contained} in $G'$.
Two ordered graphs are \emph{isomorphic} if each is contained in the other. 
Just as in the unordered setting, the \emph{t-color Ramsey number} of an ordered graph $G$, which we write as $R_t(G)$, is the minimum number of vertices of an ordered complete graph such that every edge-coloring from a set of $t$ colors contains a copy of $G$ in color $i$ for some $i$. Unless noted otherwise, we use the term \emph{Ramsey number} to refer to the 2-color Ramsey number, denoted $R(G)$. Although ``Ramsey number" usually refers to the Ramsey number for unordered graphs, in this paper we only consider ordered graphs, so $R(G)$ will always refer to the ordered setting. 

Ramsey's Theorem \cite{ramsey} implies that for $k \in \N$, there exists $n$ such that every 2-edge-colored ordered complete graph with $n$ vertices contains a monochromatic ordered complete graph with $k$ vertices. Since an ordered complete graph contains every ordered graph with the same number of vertices, and there is only one isomorphism class of ordered complete graphs with a given number of vertices, Ramsey's Theorem implies that the Ramsey number is well-defined for every ordered graph. Moreover, the Ramsey number is monotone on ordered graphs: if an ordered graph $G$ is contained in an ordered graph $H$, then $R(G) \leq R(H)$.

 The \emph{extremal number} or \emph{Tur\'an number} of a graph $H$, written $\text{ex}(n,H)$, is the maximum number of edges an $n$-vertex graph can have without containing $H$. Similarly, the \emph{extremal number} of an ordered graph $H$, which we also write as $\text{ex}(n,H)$, is the maximum number of edges an ordered $n$-vertex graph can have without containing $H$. Extremal numbers are closely related to Ramsey numbers, because if there are more than $\text{ex}(n,H)$ edges of one color, then that color must contain a monochromatic copy of $H$.  \\


Ordered Ramsey theory is a relatively recent but increasingly popular area of study. Interest in Ramsey numbers of ordered graphs arose from the well-known Erd\H{o}s-Szekeres Lemma \cite{erdosszekeres}, which states that every sequence of at least $(k-1)^2+1$ distinct integers contains a decreasing or increasing subsequence of length $k$. The \emph{monotone path} with $n$ vertices, written as $P_n^{\text{mon}}$, uses the vertex ordering $v_1, \ldots, v_n$ such that along the path the vertices are $v_1, \ldots, v_n$. 
The Erd\H{o}s-Szekeres Lemma is equivalent to $R(P_n^{\text{mon}}) = (n-1)^2+1$ \cite{FPSS}. This relationship was used to generalize the Erd\H{o}s-Szekeres Lemma to ordered Ramsey numbers in various ways \cite{FPSS, west, MS, Mub}.
Ramsey numbers of specific classes of ordered graphs were explored in \cite{balkocibulka, conlonfox, CS}.

Pach and Tardos \cite{pachtardos} studied the extremal numbers of $k$-ichromatic ordered graphs, proving that for any ordered graph $H$, the maximum number of edges that an $H$-free ordered graph with $n$ vertices can have is
$$\text{ex}(n,H) = \left(1-\frac{1}{\chi_i(H)-1}\right) \binom{n}{2}+o(n^2),$$
where $\chi_i(H)$ is the interval chromatic number of $H$. 

The \emph{alternating path} $P^{\text{alt}}_n$ uses the vertex ordering $v_1, \ldots, v_n$ such that along the path the vertices are $v_1,v_n,v_2,v_{n-1}, \ldots$. The ordered graph $P^{\text{alt}}_n$ is 2-ichromatic, because the first $\clg{n/2}$ vertices form an independent set, as do the last $\flr{n/2}$.  
Balko, Cibulka, Kr\'{a}l, and Kyn\v{c}l \cite{balkocibulka} conjectured that $P^{\text{alt}}_n$, whose Ramsey number they showed grows linearly in $n$, has the smallest Ramsey number among all ordered paths with $n$ vertices. 
The monotone path $P^{\text{mon}}_n$ has the largest interval chromatic number among all orderings of $P_n$, since every vertex must be in an interval separate from its neighbors, and above we mentioned the observation of \cite{FPSS} that $R(P^\text{mon}_n)$ is quadratic in $n$.
 
One may then think that as the interval chromatic number increases, the Ramsey number also increases. However, this is false, since Balko et al. \cite{balkocibulka} proved that for arbitrarily large $n$, there are ordered matchings $M$ on $n$ vertices such that $R(M) \geq n^{\frac{\log n}{5\log \log n}}$. An ordered matching has interval chromatic number at most $n/2$. Given such an ordered matching $M$, we can add edges to $M$ to create an ordered $n$-vertex path $P^\circ_n$ with interval chromatic number less than $n$ (avoid creating the monotone path). 
Thus by the monotonicity of the Ramsey number, $R(P^\circ_n) \geq n^{\frac{\log n}{5\log \log n}}$; the lower bound is superpolynomial in $n$.
Because $P_n^\text{mon}$ has interval chromatic number $n$ and quadratic Ramsey number, the Ramsey number on ordered paths is not monotone with respect to the interval chromatic number.
\\

We consider 2-ichromatic ordered graphs. We extend the ideas in the proof by Balko et al. \cite{balkocibulka} that the Ramsey number of the alternating path is linear in the number of vertices. We extend their lower bound methods, generalizing the alternating path to a large class of 2-ichromatic ordered graphs. The lower bound we obtain applies more generally, including many 2-ichromatic ordered graphs with many fewer edges than the alternating path. 

\begin{defn}
A $k$-ichromatic ordered graph is \emph{stitched} if the set of size $2k$ consisting of the first and last vertices from each part lies in a single component of the graph.  
\end{defn}

All connected $k$-ichromatic ordered graphs are stitched, but stitched ordered graphs need not be connected, which is the reason for introducing a new term. 
For a connected 2-ichromatic ordered graph $G$,the parts of the unique interval 2-coloring are the same as the color classes of $G$ as a bipartite graph. 
We depend heavily on the following slightly more general observation:

\begin{obs}
A stitched $2$-ichromatic ordered graph has a unique interval 2-coloring, meaning that the parts are uniquely determined.
\end{obs}

In Section \ref{mainsection}, we prove that  if $G$ is a stitched 2-ichromatic ordered graph with parts of size $m$ and $n$, then $R(G) \geq 4(\min(m,n)-1) +1$. Moreover, if $G$ also satisfies certain additional conditions, then $R(G) \geq 5(\min(m,n)-1)+1$. In Section \ref{exact}, we give exact formulas for the Ramsey numbers of two specific families of 2-ichromatic ordered graphs that are not stitched. In Section \ref{extension}, we extend our results to $t$ colors, and  we extend the upper bound given by Balko et al. \cite{balkocibulka} for a special family of ordered graphs.

\section{Matrices associated with ordered graphs}

We prove our results by converting the graph problems to matrix problems. The \emph{adjacency matrix} of an ordered graph $G$ with vertices $v_1, \ldots, v_N$ in order is an $N \times N$ $\{0,1\}$-matrix $M$ where the rows and columns are indexed by the vertices, and an entry $a_{ij}$ is $1$ if and only if $v_iv_j$ is an edge in $G$. Since we consider only undirected graphs, these matrices are symmetric along the main diagonal.

The structure of the adjacency matrix of a 2-ichromatic ordered graph is shown in Figure \ref{M}. If the parts have sizes $m$ and $n$, then the nonzero entries of the adjacency matrix are contained in two off-diagonal blocks $A_1$ and $A_2$, where $A_1$ is indexed by rows $1$ through $m$ and columns $m+1$ through $m+n$ and $A_2$ is the transpose of $A_1$. Thus any 2-ichromatic ordered graph $G$ with uniquely determined parts can be described by the single $m \times n$ matrix $A_1$, which we call the \emph{core matrix of $G$}.

\begin{figure}[h]
$$\begin{array}{c|ccc|ccc|}
& 1 & \cdots & m & m+1 & \cdots & m+n \\ \hline
1 &&&&&&\\
\vdots & & 0 & & & A_1 & \\
m &&&&&&\\ \hline
m+1 &&&&&&\\
\vdots & & A_2 & & & 0 & \\
m+n &&&&&&\\
\hline
\end{array}$$
\caption{The adjacency matrix of $A$ a 2-ichromatic ordered graph with part-sizes $n$ and $m$.}\label{M}
\end{figure}

We can view a 2-edge-coloring of the complete ordered graph $K_N$ on $N$ vertices as a symmetric $N \times N$ red/blue matrix $M$, where each entry of $M$ corresponds to an edge of $K_N$. Call an $s \times t$ submatrix $B'$ of $M$ a \emph{monochromatic copy} of an $s \times t$ $\{0,1\}$-matrix $B$ in $M$ if the nonzero entries of $B$ correspond to a set of entries in $B'$ all having the same color.

A \emph{principal submatrix} is one whose row indices and column indices are the same. If we consider a 2-ichromatic ordered graph $G$ as a subgraph of $K_N$, a copy of the adjacency matrix $A$ of $G$ will appear as a principal submatrix $A'$ of $M$. Let $A_1'$ and $A_2'$ be the corresponding copies of submatrices $A_1$ and $A_2$ of $A$ (as shown in Figure \ref{M}). Due to the ordering of $G$, $A_1'$ is above the diagonal and is completely above and to the right of every entry in $A_2'$. 
\\

In Section \ref{extension} we make use the analogue for matrices of Tur\'{a}n numbers. A $\{0,1\}$-matrix $M$ \emph{contains} an $r \times s$ $\{0,1\}$-matrix $A$ if $M$ has a $r \times s$ submatrix $A'$ that has a 1 in every position where $A$ has a 1, with the rows and columns in the same order as in $A$. A matrix \emph{avoids} $A$ if it does not contain $A$. A matrix $A$ is \emph{tightly contained} in a matrix $M$ if $M$ contains $A$ and has the same dimensions as $A$.
The \emph{extremal number} $\text{ex}_A(m,n)$ of a $\{0,1\}$-matrix $A$ is the maximum number of nonzero entries in an $m \times n$ $\{0,1\}$-matrix avoiding $A$. An $r \times s$ matrix $A$ is \emph{minimalist} if $\text{ex}_A(m,n) = (s-1)m+(r-1)n-(r-1)(s-1)$. The term ``minimalist" was introduced by F\"{u}redi and Hajnal \cite{furedihaijnal} because for each choice of $m$ and $n$, the value $(s-1)m+(r-1)n-(r-1)(s-1)$ is the smallest extremal number of any $r \times s$ $\{0,1\}$-matrix $A$ having at least one nonzero entry.
We say also that a 2-ichromatic ordered graph $G$ is \emph{minimalist} if the core matrix of $G$ is minimalist. 

One can build larger minimalist matrices from existing minimalist matrices by a method of F\"{u}redi and Hajnal \cite{furedihaijnal}. 
An elementary operation on a $\{0,1\}$-matrix $M$ produces a $\{0,1\}$-matrix $M'$
by adding a new first or last row or column containing a single 1 next to a 1 of $M$.
F\"{u}redi and Hajnal proved that if $M$ is a minimalist matrix and $M'$ is obtained from $M$ by applying elementary operations, then $M'$ is minimalist. 
Since the $1 \times 1$ identity matrix is clearly minimalist, this is one way to create minimalist matrices. Also, any nonzero matrix tightly contained in a minimalist matrix is minimalist.
However, these are not all the minimalist matrices; the matrix below is minimalist but neither is created by applying elementary operations to the $1 \times 1$ identity matrix nor is tightly contained in such a matrix \cite{furedihaijnal}.

\[\begin{bmatrix}
0 & 1 & 0 \\
1 & 0 & 0 \\
1 & 0 & 1
\end{bmatrix}\]

The ordered path $P^{\text{alt}}_n$ is minimalist. Balko et al. \cite{balkocibulka} used this fact in proving an upper bound on $R(P^{\text{alt}}_n)$. In Section \ref{extension} we extend their result to $t$ colors and all minimalist 2-ichromatic ordered graphs.


\section{Lower bounds for stitched 2-ichromatic ordered graphs}\label{mainsection}

\begin{thm} Let $G$ be a stitched $2$-ichromatic ordered graph. If the parts have size $m$ and $n$, then $R(G) \geq 4r+1$, where $r = \min(m,n)-1$. \label{mainthm}
\end{thm}

\begin{proof} Let $A$ be the adjacency matrix of $G$, and let $A_1$ and $A_2$ be the core matrix of $G$ and its transpose, respectively.
To prove the lower bound, we construct a symmetric $N\times N$ red/blue matrix $M$ such that no $(m +n) \times (m+n)$ principal submatrix $A'$ is a monochromatic copy of $A$ in $M$. This corresponds to a 2-coloring of the complete ordered $K_N$ that does not contain a monochromatic copy of $G$.
Since $M$ is symmetric and $A'$ is principal, it suffices to avoid monochromatic copies of $A_1$ above the diagonal of $M$. 
We achieve this by block coloring: when $N$ is a multiple of $r$, $M$ is \emph{$r \times r$ block 2-colored} when each $r \times r$ block of entries is monochromatic.

Suppose that $M$ is $r\times r$ block 2-colored and has a principal monochromatic copy $A'$ of $A$.
Let $A_1'$ and $A_2'$ be the monochromatic copies of $A_1$ and $A_2$ in $A'$.
Recall that $A_1'$ is on or above the diagonal and is completely above and to the right of every entry in $A_2'$.
Also $A_2'$ is located below the diagonal in the positions reflecting $A_1'$. 

Let $X$ denote the set of first and last vertices of the parts in the interval 2-coloring of $G$.
Since $G$ is a stitched 2-ichromatic ordered graph, the four vertices of $X$ are connected by paths. 
Consider only those entries in $A_1$ that correspond to the edges of these paths; let $C$ be the set of these entries. 
These edges form a connected subgraph of $G$, and there exists at least one entry in $C$ in each of the first and last columns and the first and last rows of $A_1$. 
Let $C'$ be the corresponding entries in $A_1'$.

Now consider the $r \times r$ blocks of $M$ that contain entries in $C'$; recall that $r = \min(m,n)-1$.
Since $A'$ is a principal monochromatic copy of $A$, these blocks have the same color in $M$.
Since $C'$ has entries in the extreme rows and columns of $A_1'$, and $A_1'$ spans $m$ rows and $n$ columns, $C'$ cannot be contained within a single row or column of blocks. 
Because the subgraph corresponding to $C'$ is connected, there is a ``path of blocks" joining any two blocks used by $C'$, meaning that at each step only the row or only the column is changed.
Since $C'$ cannot be captured in a single row or column of blocks, there is thus such a path with two steps, one row change and one column change, which yields one of the four patterns in Figure \ref{patt}.
(Note that the blocks in a pattern need not be contiguous.) 
We conclude that if an $r \times r$ block 2-colored matrix $M$ does not have a monochromatic instance of one of these four block patterns on or above the diagonal, then we will not have a monochromatic copy of $A$. 

\begin{figure}[h]
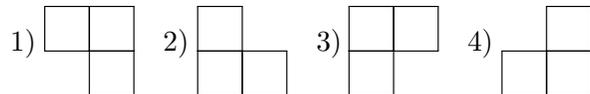

$$1)\; \ydiagram{0+2,1+1}\quad 
2)\; \ydiagram{1,2} \quad 
3)\: \ydiagram{2,1} \quad
4)\; \ydiagram{1+1, 2} $$
\caption{Forbidden $r\times r$ block patterns in $M$}\label{patt}
\end{figure}

We claim further that the three blocks we have extracted from $C'$ cannot occur in pattern 1 using a diagonal block, and hence we do not need to avoid such instances of pattern 1 to establish the lower bound. Consider such an instance.
If one of these three blocks is on the diagonal, then it must be the upper left or lower right block. By symmetry we may assume it is the upper left block, let this be $B$.
Let $x$ be an entry of $C'$ in $B$, and let $y$ be an entry of $C'$ in the lower right block. 
Because $A'$ is principal, the reflection of $C'$ through the diagonal is contained in $A_2'$. When we reflect $x$ to $\hat{x} \in A_2'$, $\hat{x}$ is still in $B$, but the original entry $y \in A_1'$ is in a lower row of blocks. 
Hence the entry $y \in A_1'$ is not above the entry $\hat{x} \in A_2'$, which contradicts the structure of matrix $A$ shown in Figure \ref{M}. 

When $N = 4r$, the block coloring of an upper triangular matrix shown in Figure \ref{mainavoid} avoids the four patterns in Figure \ref{patt}, except for an instance of pattern 1 using the diagonal. For later use, we may say that the color used on the upper right block is color 1 (red). Therefore, $M$ avoids $A$, and the corresponding 2-edge-coloring of the ordered complete graph on $4r$ vertices avoids the ordered graph $G$.
\end{proof}

\begin{figure}[h]
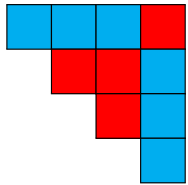
\centering
$\begin{ytableau}
*(cyan) & *(cyan) & *(cyan) & *(red) \\
\none  & *(red) & *(red) & *(cyan) \\
\none  & \none  & *(red) & *(cyan) \\
\none  & \none  & \none  & *(cyan)
\end{ytableau}$
\caption{Block coloring that avoids the forbidden patterns} \label{mainavoid}
\end{figure}


The lower bound of $4r+1$ in Theorem \ref{mainthm} relies on the necessity of avoiding monochromatic instances of all four block patterns in Figure \ref{patt}. We can improve the bound if avoiding the specific graph only requires us to avoid certain subsets of these patterns: If we only need to avoid pattern 2, then we improve the lower bound to $5r+1$, and we can also obtain a lower bound of $5r+1$ when avoiding just patterns 3 and 4. The block colorings that prove these claims are shown in Figure \ref{avoid}. The coloring in Figure \ref{avoid}a appears in \cite{balkocibulka}.

\begin{figure}[h]
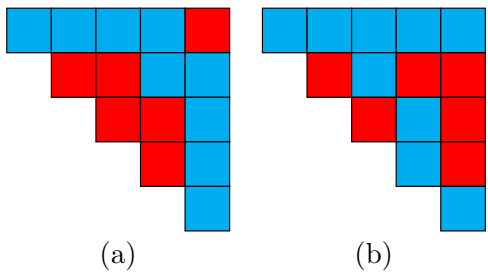
\centering
\begin{tabular}{cc}
$\begin{ytableau}
*(cyan) & *(cyan) & *(cyan) & *(cyan) & *(red) \\
\none  & *(red) & *(red)& *(cyan) & *(cyan) \\
\none  & \none  & *(red) & *(red) & *(cyan) \\
\none & \none  & \none  & *(red) & *(cyan) \\
\none  & \none  & \none  &  \none & *(cyan) 
\end{ytableau}$  & 
 $\begin{ytableau}
*(cyan) & *(cyan) & *(cyan) & *(cyan) & *(cyan) \\
\none  & *(red) & *(cyan) & *(red) & *(red) \\
\none  & \none  & *(red) & *(cyan) & *(red) \\
\none  & \none  & \none  & *(cyan) & *(red) \\
\none  & \none  & \none  &  \none & *(cyan) 
\end{ytableau}$ \\
(a) & (b)
\end{tabular}
\caption{Block colorings that (a) avoid patterns 3 and 4, and (b) avoid pattern 2. }\label{avoid}
\end{figure}

Motivated by this, we determine some of the graphs that require either pattern 3 or 4 to cover the nonzero entries in their matrix, as well as some of the graphs that require pattern 2. Thus we obtain lower bounds on Ramsey numbers for these ordered graphs. 

Corollary \ref{altcor} directly extends the result of \cite{balkocibulka} that the Ramsey number of the alternating path is at least $5r +1$, where $r$ is one less than half the number of vertices in the path \cite{balkocibulka}, to a larger family of graphs containing the alternating path.


\begin{cor}\label{altcor} Let $G$ be a stitched 2-ichromatic ordered graph with parts $v_1, \ldots, v_m$ and $v_{m+1}, \ldots, v_{m+n}$. If $G$ has edges $v_1v_{m+n}$ and $v_mv_{m+1}$, then $R(G) \geq 5r+1$, where $r = \min(m,n)-1$.
\end{cor}

\begin{proof} Let $A_1$ be the core matrix of $G$, and as in Theorem \ref{mainthm} consider an $r \times r$ block 2-colored matrix $M$, which contains a copy $A_1'$ of $A_1$ above the diagonal. Edges $v_1v_{m+n}$ and $v_mv_{m+1}$ of $G$ correspond to nonzero entries in the top right corner and lower left corner of $A_1$. 
Since $A_1'$ spans at least $m$ rows and $n$ columns, one cannot cover these two entries with only blocks in the same row or column.
Hence the extreme blocks intersected by $A_1'$ must be two diagonally placed blocks, $B_1$ to the upper right and $B_2$ to the lower left. 

Since $G$ is stitched, it has paths connecting the vertices $\{v_1, v_m, v_{m+1}, v_{m+n}\}$. Let $C$ be the collection of entries of $A_1$ that correspond to the edges of these paths, and let $C'$ be the corresponding entries of $A_1'$.
As in the proof of Theorem \ref{mainthm}, since the graph corresponding to $C'$ is connected, there is a path of blocks joining any two blocks used by $C'$. 
Since $B_1$ is the top and rightmost block used by $C'$, and there must be a path of blocks to reach $B_2$, the path first steps downward or leftward, and when the path first takes a step in the other direction, we obtain pattern 3 or pattern 4. 
The $r \times r$ block 2-colored $5r \times 5r$ matrix shown in Figure \ref{avoid}a avoids pattern 3 and pattern 4 above the diagonal. Thus it yields a 2-edge-coloring of the ordered complete graph on $5r$ vertices that does not contain a monochromatic copy of the ordered graph $G$.
\end{proof}


Corollary \ref{lacecor} proves the same lower bound for a different family of ordered graphs.

\begin{cor}\label{lacecor} Let $G$ be a 2-ichromatic ordered graph with parts $v_1, \ldots, v_m$ and $v_{m+1}, \ldots, v_{m+n}$. If $G$ has edges $v_1v_{m+1}$, $v_mv_{m+n}$, and $v_mv_{m+1}$, then $R(G) \geq 5r+1$, where $r = \min(m,n)-1$. 
\end{cor}

\begin{proof} Let $A$ be the core  matrix of $G$. The specified edges of $G$ correspond to nonzero entries in the top left, bottom left, and bottom right corners of $A$. Since $r = \min(m,n)-1$, one must use two blocks in one column of blocks to cover the top left entry and bottom left entry of $A$, and then use a third block to the right of the bottom block in the same row of blocks to cover the lower right entry. Thus, to cover the nonzero entries with $r \times r$ blocks, one must use pattern 2 from Figure \ref{patt}. The $r \times r$ block 2-colored $5r \times 5r$ matrix shown in Figure \ref{avoid}b avoids pattern 2 on or above the diagonal. Thus it yields an edge-coloring of the ordered complete graph on $5r$ vertices that does not contain a monochromatic copy of $G$. 
\end{proof}

\begin{remark} By case analysis, we can show that it is not possible to 2-color a triangular configuration with more than five rows and columns of blocks to avoid pattern 2 or avoid patterns 3 and 4. It is also not possible to 2-color a triangular configuration with more than four rows and columns of blocks to avoid any other subset of the four patterns in Figure \ref{patt}.
Thus we cannot enforce larger lower bounds than $4r+1$ with this block coloring technique for any other subset of the four block patterns, and we cannot enforce larger lower bounds than $5r+1$ for the two families in Corollaries \ref{altcor} and \ref{lacecor}. 
\end{remark}


\section{Exact Values for Special Families}\label{exact}

The lower bounds presented in Corollaries \ref{altcor} and \ref{lacecor} are larger than the Ramsey numbers for the ordered graphs consisting of only the matching $M$ consisting of $\{v_1v_{m+n}, v_mv_{m+1}\}$, even though they differ by perhaps only one edge. 
For example, an ordered graph treated by Corollary \ref{altcor} can be obtained by adding just one edge to $M$ to make the graph stitched.
We consider two specific examples among 2-ichromatic ordered graphs with color classes of sizes $m$ and $n$. 

\begin{figure}[h]\centering
\begin{tabular}{cc}
\includegraphics[scale=.3]{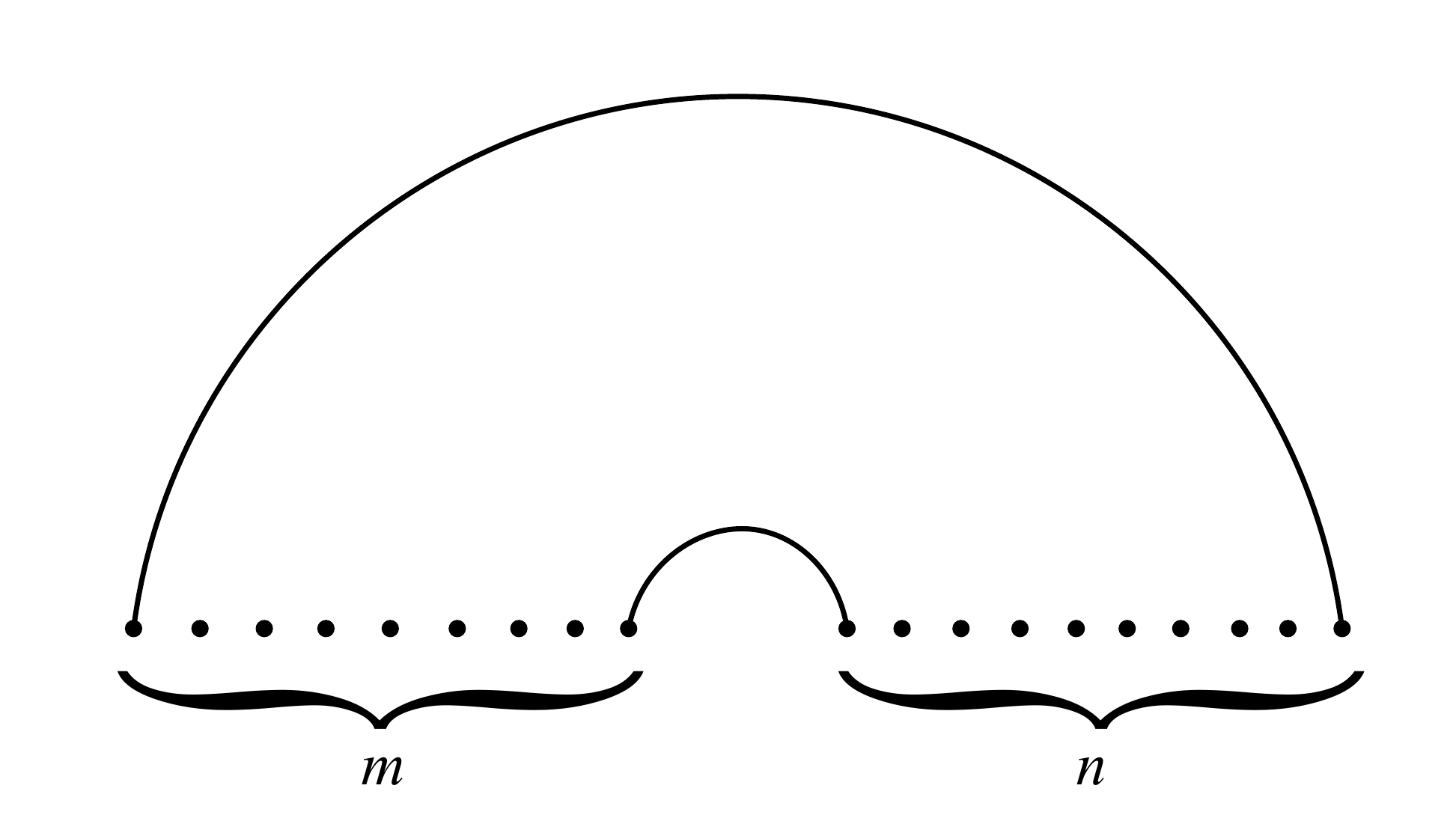} &
\includegraphics[scale=.3]{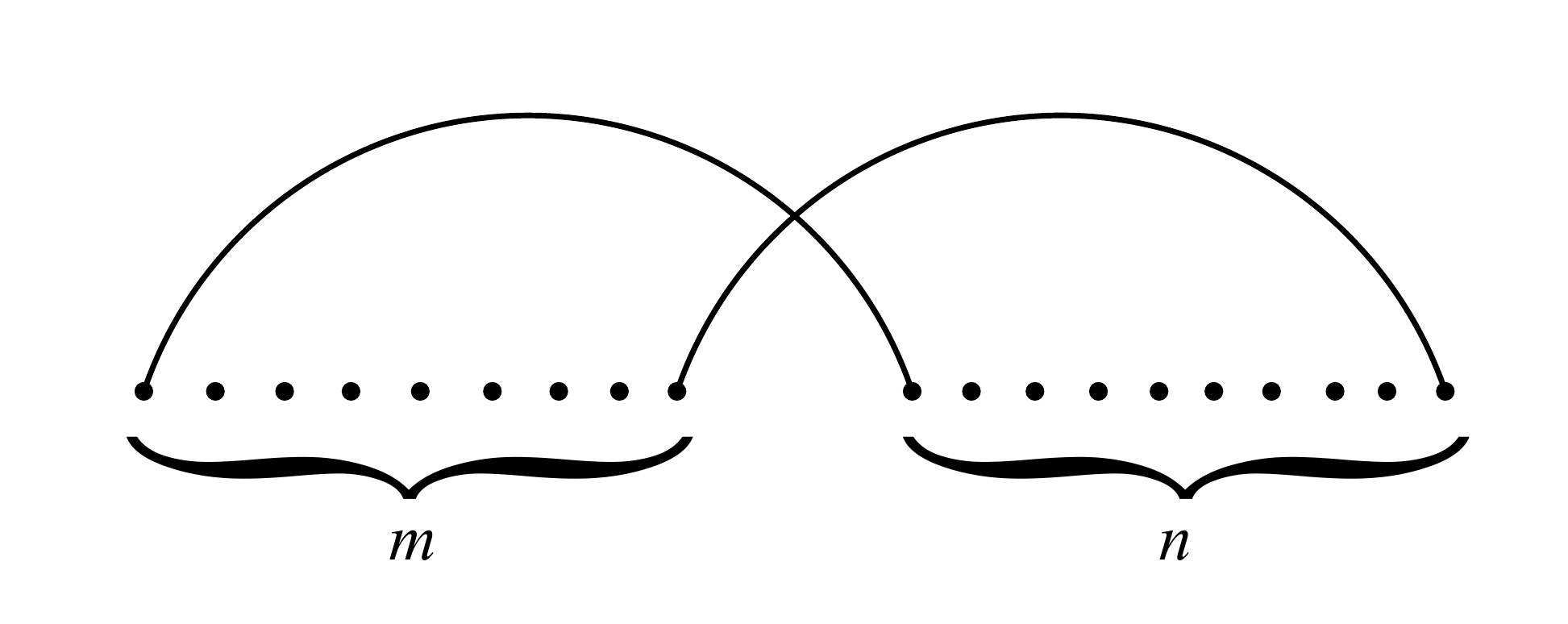} \\
(a) & (b)
\end{tabular}
\caption{2-ichromatic ordered graphs with (a) nested edges or (b) crossing edges.} \label{graphs}
\end{figure}


\begin{prop} \label{nestprop}
If $G$ is the $2$-ichromatic ordered graph with parts of size $m$ and $n$ and two nested edges, one joining the outermost vertices and one joining the innermost vertices (see Figure \ref{graphs}a), then $R(G) = 2m+2n-2$.
\end{prop}

\begin{proof}
\emph{Upper Bound.}
Consider a 2-edge-coloring of the ordered complete graph whose $2m+2n-2$ vertices are $v_1, \ldots, v_{2m+2n-2}$. Consider the three edges $v_1 v_{2m+2n-2}$, $v_m v_{2m+n-1}$, and $v_{2m-1}v_{2m}$, as shown in Figure \ref{proofgraph1}. Two of these three edges have the same color, and these two edges along with $m-2$ vertices between their left endpoints and $n-2$ vertices between their right endpoints form a copy of $G$. Thus $R(G) \leq 2m+2n-2$.

\begin{figure}[h]\centering
\vcenteredinclude{scale=.4}{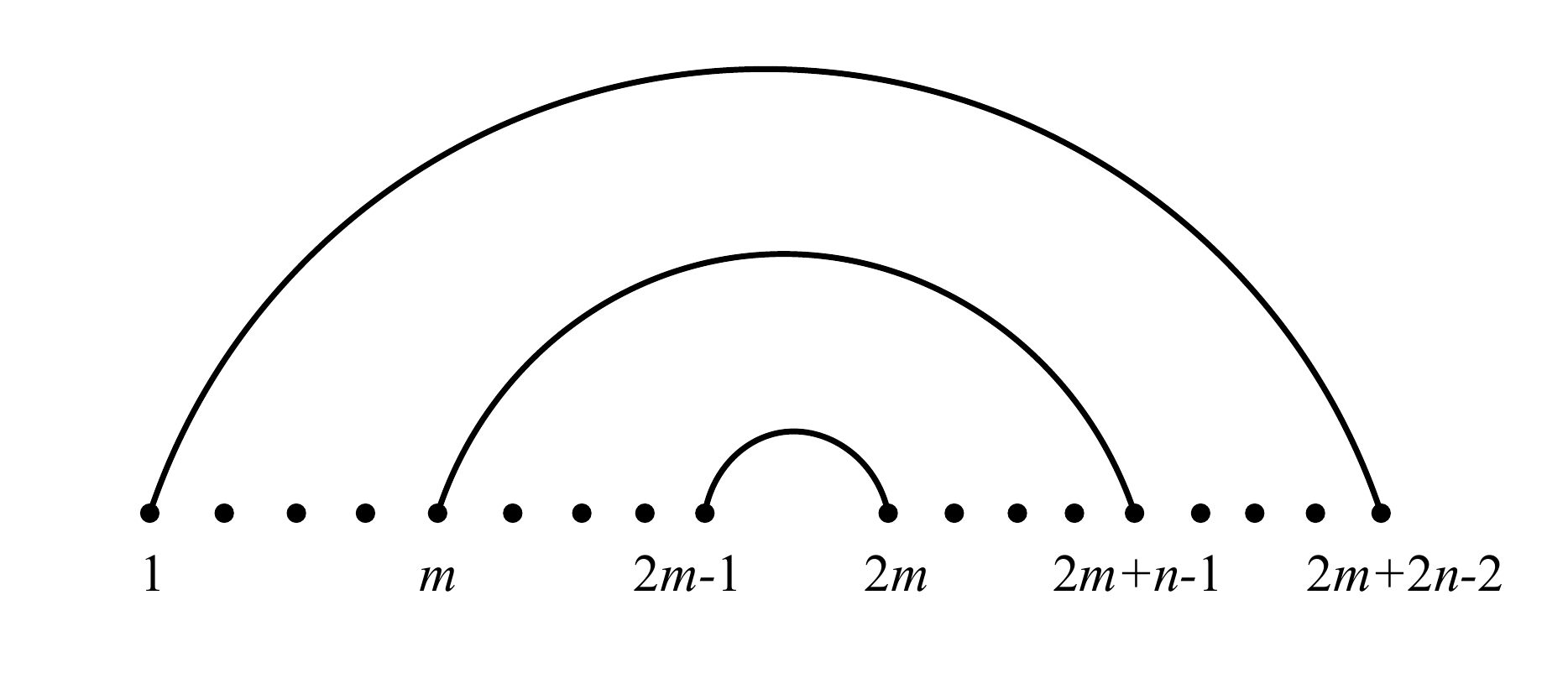}
\caption{Particular edges in an ordered graph with $2m+2n-2$ edges} \label{proofgraph1}
\end{figure}

\emph{Lower Bound.} 
Given the ordered complete graph with $2m+2n-3$ vertices, we construct a coloring avoiding $G$. 
Assign red to each edge having an endpoint among either the first $m-1$ vertices or the last $n-1$ vertices; assign blue to all other edges. 
We claim that this coloring avoids $G$. Any two red nested edges have left endpoints at most $m-2$ apart from each other or right endpoints at most $n-2$ apart from each other, so there is no red $G$. All blue edges are contained within a subgraph having only $m+n-1$ vertices, so there is no blue $G$. Thus $R(G) > 2m+2n-3$. 
\end{proof}


\begin{prop}\label{crossprop}
If $G$ is the $2$-ichromatic ordered graph with parts of size $m$ and $n$ and two crossing edges, one joining the first vertices in each part and one joining the last vertices in each part (see Figure \ref{graphs}b), then $R(G) = m+n+\max(m,n)-1$.
\end{prop}

\begin{proof}  
We may assume $m \geq n$ (otherwise, reverse the ordering).

\emph{Upper Bound.}
Consider a 2-edge-coloring of the ordered complete graph with $2m+n-1$ vertices $v_1, \ldots v_{2m+n-1}$ in order.
Consider two long edges ($v_1v_{2m}$ and $v_mv_{2m+n-1}$) and three shorter edges ($v_1v_{m+1}$, $v_{m}v_{2m}$, and $v_{2m-1}v_{2m+n-1}$), as shown in Figure \ref{proofgraph2}.
Three of these five edges have the same color. 
The two long edges together with the first $m-2$ and last $n-2$ other vertices form a copy of $G$. 
Any two consecutive short edges and the vertices between their endpoints form a copy of $G$. 
If one long edge and the two nonconsecutive short edges have the same color, then the long edge and the short edge incident to the other extreme vertex form a copy of $G$, together with the first $m-2$ and last $n-2$ other vertices. Therefore, there is a monochromatic copy of $G$, and $R(G) \leq 2m+n-1$.

\begin{figure}[h]\centering
\vcenteredinclude{scale=.4}{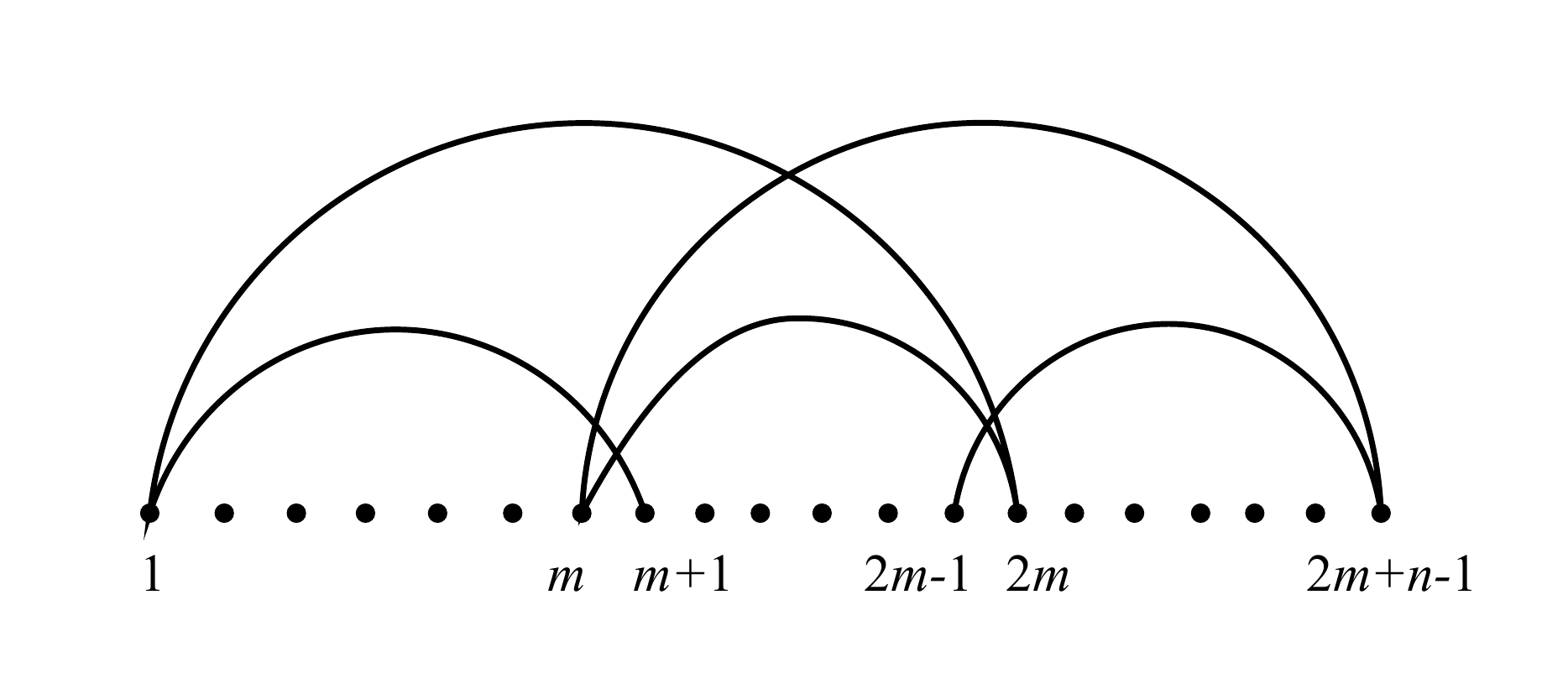}
\caption{Particular edges in an ordered graph with $2m+n-1$ edges.} \label{proofgraph2}.
\end{figure}

\emph{Lower Bound.}
Given the ordered complete graph with $2m+n-2$ vertices, we construct a coloring avoiding $G$. 
Assign red to each edge having an endpoint among the first $m-1$ vertices; assign blue to all other edges. 
We claim that this coloring avoids $G$. 
Any two red crossing edges have their left endpoints at most $m-2$ apart from each other, so there is no red $G$, and all blue edges are contained within a subgraph having only $m+n-1$ vertices, so there is no blue $G$.  Thus $R(G) > 2m+n-2$. 
\end{proof}

Thus, in the case $m=n$, adding one specific edge to either nested edges or crossing edges increases the Ramsey number. We can add an edge to the nested edges from Proposition \ref{nestprop} to make the graph stitched, increasing the Ramsey number from $4n-2$ to at least $5n-4$. Similarly, we can add the edge joining the last vertex in the first part and the first vertex in the second part to the crossing edges from Proposition \ref{crossprop} to make the graph satisfy the hypothesis of Corollary \ref{lacecor}, increasing the Ramsey number from $3n-1$ to at least $5n-4$.


\section{Extensions to $t$-color Ramsey Numbers}\label{extension}

We can also extend our general lower bounds to the $t$-color Ramsey case, where we color the edges from a set of $t$ colors:


\begin{cor}
\begin{enumerate}
\item If $G_1$ is a stitched 2-ichromatic ordered graph with parts of sizes $n$ and $m$, then $R_t(G_1)\geq 2tr+1$, where $r = \min(n,m)-1$.

\item Let $G_2$ be a stitched 2-ichromatic ordered graph with parts $v_1, \ldots, v_m$ and $v_{m+1}, \ldots, v_{m+n}$. If edges $v_1v_{m+n}$ and $v_mv_{m+1}$ are contained in $G_2$, then $R_t(G_2) \geq (2t+1)r+1$, where $r = \min(m,n)-1$.

\item Let $G_3$ be a 2-ichromatic ordered graph with parts $v_1, \ldots, v_m$ and $v_{m+1}, \ldots, v_{m+n}$. If edges $v_1v_{m+1}$, $v_mv_{m+n}$, and $v_mv_{m+1}$ are contained in $G$, then $R_t(G_3) \geq (2t+1)r+1$, where $r = \min(m,n)-1$. .
\end{enumerate}
\end{cor} 

\begin{proof} Using $t=2$ and Theorem \ref{mainthm}, Corollary \ref{altcor}, and Corollary \ref{lacecor} as the base cases for our three statements, we use induction on $t$. Having already colored the appropriately sized matrix for $t$ colors, we can extend the coloring to $t+1$ colors by adding new rows and columns as in Figure \ref{tcolor}: For statement 1, add a top row and right column of the new color, with one block of color 1 in the top right. For statement 2, add a top row and right row. For statement 3, add a top row and diagonal.

The first matrix coloring (Figure \ref{tcolor}a) avoids monochromatic instances of all four patterns in Figure \ref{patt} (except pattern 1 on the diagonal), the second (Figure \ref{tcolor}b) avoids patterns 3 and 4, and the third (Figure \ref{tcolor}c) avoids pattern 2.  Thus, they avoid $G_1$, $G_2$, and $G_3$ respectively. Each inductive step adds two rows and columns of blocks, and the results follow.
\end{proof}

\begin{figure}[h]
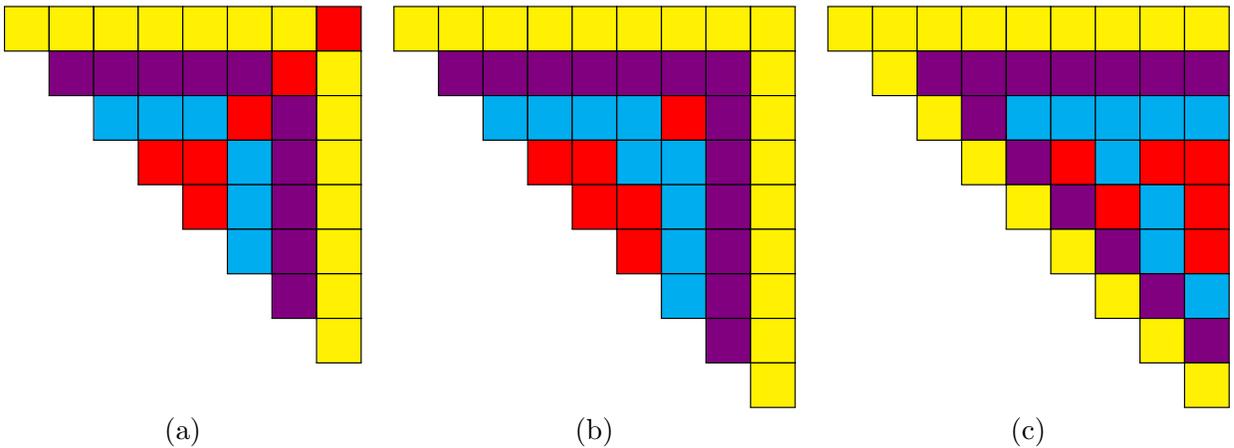
\centering
\begin{tabular}{ccc}
$\begin{ytableau}
*(yellow) & *(yellow) & *(yellow) & *(yellow) & *(yellow) & *(yellow) & *(yellow) & *(red)\\
\none & *(violet) & *(violet) & *(violet)& *(violet) & *(violet) & *(red) & *(yellow)\\
\none & \none & *(cyan) & *(cyan) & *(cyan) & *(red) & *(violet) & *(yellow) \\
\none & \none & \none  & *(red) & *(red)& *(cyan) & *(violet)& *(yellow) \\
\none & \none & \none  & \none  & *(red) & *(cyan) & *(violet) & *(yellow)\\
\none & \none & \none  & \none  & \none  & *(cyan) & *(violet) & *(yellow)\\
\none & \none & \none  & \none  & \none  &  \none & *(violet) & *(yellow) \\
\none & \none & \none  & \none  & \none  &  \none & \none & *(yellow)
\end{ytableau}$
&
$\begin{ytableau}
*(yellow) & *(yellow) & *(yellow) & *(yellow) & *(yellow) & *(yellow) & *(yellow) & *(yellow) & *(yellow)\\
\none & *(violet) & *(violet) & *(violet) & *(violet) & *(violet) & *(violet) & *(violet) & *(yellow)\\
\none & \none & *(cyan) & *(cyan) & *(cyan) & *(cyan) & *(red) & *(violet) & *(yellow)\\
\none & \none & \none  & *(red) & *(red)& *(cyan)& *(cyan) & *(violet) & *(yellow) \\
\none & \none & \none  & \none  &*(red) & *(red) & *(cyan) & *(violet) & *(yellow)\\
\none & \none & \none  & \none  & \none  & *(red) & *(cyan) & *(violet) & *(yellow)\\
\none & \none & \none  & \none  & \none  &  \none & *(cyan) & *(violet) & *(yellow) \\
\none & \none & \none  & \none  & \none  &  \none & \none & *(violet) & *(yellow)\\
\none & \none & \none  & \none  & \none  &  \none & \none & \none & *(yellow)
\end{ytableau}$ 
&
$\begin{ytableau}
*(yellow) & *(yellow) & *(yellow) & *(yellow) & *(yellow) & *(yellow) & *(yellow) & *(yellow) & *(yellow)\\
\none & *(yellow) & *(violet) & *(violet)  & *(violet)  & *(violet)  & *(violet)  &*(violet)&*(violet)\\
\none & \none & *(yellow) & *(violet) & *(cyan) & *(cyan) & *(cyan) & *(cyan) & *(cyan) \\
\none & \none & \none  & *(yellow) & *(violet) & *(red)& *(cyan) & *(red) & *(red) \\
\none & \none & \none  & \none  &*(yellow) & *(violet) & *(red) & *(cyan) & *(red)\\
\none & \none & \none  & \none  & \none  & *(yellow) & *(violet) & *(cyan) & *(red)\\
\none & \none & \none  & \none  & \none  &  \none & *(yellow) & *(violet) & *(cyan) \\
\none & \none & \none  & \none  & \none  &  \none & \none & *(yellow) & *(violet)\\
\none & \none & \none  & \none  & \none  &  \none & \none & \none & *(yellow)
\end{ytableau}$ \\
(a) & (b) & (c)
\end{tabular}
\caption{Block colorings for the 4-color Ramsey extensions.} \label{tcolor}
\end{figure}



Since we have thus far found lower bounds that are linear in the number of vertices for certain 2-ichromatic ordered graphs, it is desirable to also find linear upper bounds for some 2-ichromatic ordered graphs. In proving that the alternating path has linear Ramsey number, Balko, Cibulka, Kr\'{a}l and Kyn\v{c}l \cite{balkocibulka} used minimalist matrices to prove a linear upper bound for the alternating path $P_n^{\text{alt}}$ with $n$ vertices:
$$R(P_n^{\text{alt}}) \leq 2n-4 + \sqrt{2n^2-8n+11}$$ 
In fact, their proof is valid for any minimalist graph with equal part-sizes.
Here we extend this proof to $t$ colors and to 2-ichromatic ordered graphs with general part-sizes:

\begin{prop} If $G$ is a minimalist 2-ichromatic ordered graph with parts size $m$ and $n$, then 
$$R_t(G) \leq t(n+m-2) + \sqrt{t^2(n+m-2)^2 + 2t(3(n+m)-4-2mn)}.$$
\end{prop}

\begin{proof} Consider an ordered complete graph with $N$ vertices having a $t$-color edge-coloring that contains no monochromatic $G$.
Consider the adjacency matrix for the coloring, and consider the upper right submatrix $B$ formed by the first $\clg{N/2}$ rows and last $\flr{N/2}$ columns. Because the given coloring has no monochromatic copy of $G$, this matrix $B$ does not contain a monochromatic copy of the core matrix of $G$. Since $G$ is minimalist,
\begin{align*}
\text{ex}_{A_1}(\clg{N/2}, \flr{N/2}) &= (n-1)\clg{N/2} + (m-1)\flr{N/2}-(m-1)(n-1) \\
&\leq \frac{1}{2}(n+m)(N+3) - N - 2 - mn \\
&\leq \left(\frac{1}{2}(n+m) -1\right)N + \left(\frac{3}{2}(n+m) -2-mn \right).
\end{align*}
By the pigeonhole principle, at least $\frac{1}{t}$ of the $\clg{N/2}\flr{N/2}$ edges in $B$ have the same color. Note that $\frac{1}{t}\clg{N/2}\flr{N/2} \geq (N^2-1)/4t$.
To avoid monochromatic copies of $G$, the following inequality must be satisfied.
\begin{align*}
N^2-1 &\leq (2t(n+m)-4t)N + 6t(n+m)-8t-4tmn.
\end{align*}
Now the quadratic formula completes the proof.
\begin{align*}
N &\leq t(n+m-2) + \sqrt{(t(n+m)-2t)^2 + 6t(n+m)-8t-4tmn}.
\qedhere
\end{align*}
\end{proof}

Thus there is a family of 2-ichromatic ordered graphs whose Ramsey number is linear in the number of vertices. 
However, this linear upper bound is only for minimalist 2-ichromatic ordered graphs. 
Finding upper bounds for more 2-ichromatic ordered graphs is a topic for further study.
For example, whenever the extremal number of the core matrix of a 2-ichromatic ordered graph $G$ is subquadratic in the number of vertices, one can use the method above to obtain an upper bound for $R_t(G)$.


\end{document}